\documentclass{amsart}
\usepackage{amsmath, amssymb, epsfig, graphicx, color }
\usepackage[all]{xy}
\usepackage{enumerate,comment}
\newtheorem{thm}{Theorem}[section]

\newtheorem{lem}[thm]{Lemma}

\newtheorem*{repeatthm}{Theorem \ref{thm:repeat}}
\newtheorem*{propnonum}{Proposition}
\newtheorem*{cornonum}{Corollary}
\newtheorem*{thmnonum}{Theorem}
\newtheorem*{factnonum}{Fact}
\newcommand{\p}{{\mathbb P}}
\newcommand{\nat}{{\mathbb N}}
\newcommand{\baire}{{ \nat ^{\nat}}}
\newcommand{\com}{{\mathcal K}}
\newcommand{\g}{{\mathcal G}}
\newcommand{\V}{{\mathcal V}}

\newcommand{\T}{\mathcal T}
\newcommand{\s}{\mathcal S}

\title{ $\p$-domination and Borel sets }
\author[basile]{D\'esir\'ee Basile}
\author[Darji] {Udayan B.~Darji}
\email[Basile]{de.basile@gmail.com}
\email[Darji]{ubdarj01@louisville.edu}
\address[Basile]{Dipartimento di matematica e informatica, viale Andrea Doria 6, 95125 Catania, Italia.}
\address[Darji]{Department of Mathematics, University of Louisville, Louisville, KY 40292, USA.}
\thanks{The second author thanks the hospitality of Universit\`a delgi Studi di Catania
and acknowledges support from GNSAGA of INDAM}
\begin{document}
\begin{abstract}
In recent years much attention has been enjoyed by the topological spaces which are dominated by
second countable spaces. The origin of the concept dates back to the 1979 paper of Talagrand in
which it was shown that for a compact space $X$, $C_p(X)$ is dominated by $\p$, the set of irrationals, if and only if 
$C_p(X)$ is $K$-analytic. Cascales extended this result to spaces $X$ which are angelic and finally
in 2005 Tkachuk proved that the Talagrand result is true for all Tychnoff spaces $X$. In recent years,
the notion of $\p$-domination has enjoyed attention independent of $C_p(X)$. In particular, Cascales,
Orihuela and Tkachuk proved that a Dieudonn\`e complete space is K-analytic if and only if it is dominated 
by $\p$. A notion related to $\p$-domination is that of strong $\p$-domination. Christensen had earlier
shown that a second countable space is strongly $\p$-dominated if and only if  it is completely metrizable. We show
that a very small modification of the definition of $\p$-domination characterizes Borel subsets of Polish spaces.\end{abstract}
\maketitle
\section{introduction}
All spaces considered in this note are Tychnoff, i.e., completely regular and Hausdorff. If $X$ is a space,
then $\com(X)$ denotes the set all of  compact subsets of $X$  endowed with the Hausdorff  topology.
We let $\p$ denotes the set of irrational numbers with the topology it inherits from the reals. As we are only
interested in the topological properties of $\p$ in this note, for all practical purposes $\p$ is simply the Baire
space $\baire$.

The topic of this note dates back to a 1979 paper of Talagrand in functional analysis \cite{talagrand}.
In particular, he showed that if $X$ is a compact space then $C_p(X)$ is $K$-analytic if and only if  there exists a collection of 
compact set $\{A_p : p \in \p\}$  whose union is $X$ and satisfies the property that $A_p \subseteq A_q$ whenever $p < q$. 
Cascales \cite{c} extended this result by showing the above holds for spaces $X$ which are angelic and finally
in 2005 Tkachuk \cite{tkachuk} showed that the result holds for all Tychnoff spaces $X$.  Moreover, in his paper
Tkachuk initiated a systematic study of the spaces which are now called spaces with {\em $\p$-directed covers}, i.e., those spaces $X$
for which there exists a cover of compact sets $\{A_p: p \in \p\}$ such that $p < q$ implies $A_p \subseteq A_q$.

In 2011, Cascales, Orihuela and Tkachuk  \cite{cot} initiated the study of a related concept. If $X,M$ are topological
spaces, following \cite{cot}, we say that {\em $X$ is $M$-dominated}  if there exists a collection $\{A_K: K \in \com(M)\}$
of compact subsets of $X$ such that $\bigcup_{K }A_K =X$ and $K \subseteq L$ implies that $A_K \subseteq A_L$.
Among many other results in the paper, the following was proved.
\begin{propnonum}(Proposition~2.2, \cite{cot}) The following conditions are equivalent for any space $X$.
\begin{itemize}
\item $X$ has a $\p$-directed cover.
\item $X$ is $\p$-dominated.
\item $X$ is dominated by some Polish space.
\end{itemize}
\end{propnonum}
Using the above proposition and an earlier result in \cite{co}, they obtained following corollary.
\begin{cornonum}\label{cor:cot}(Corollary 2.3, \cite{cot})
A Dieudonn\'e complete space is $K$-analytic if and only if  it is dominated by a Polish space.
\end{cornonum}

In a very early paper, Christensen \cite{christensen}  proved a result concerning completely metrizability
of a second countable space. In the current terminology we say, following
\cite{cot}, that {\em space $X$ is strongly $M$-dominated} if $X$ is $M$-dominated by
a cover $\{A_K: K \in \com(M)\}$ such that each compact subset of $X$ is a subset of some $A_K$.

\begin{thmnonum}\label{thm:christensen} (Theorem 3.3, \cite{christensen})
A second countable space is strongly $M$-dominated by a Polish space if and only if it is completely metrizable.
\end{thmnonum}

Let us make some observations now. Suppose for the moment that $X$ is a separable metric space. Recall that in separable metric
spaces the notions of $K$-analytic and analytic coincide.  Hence, Corollary~\ref{cor:cot}
says that $X$ is dominated by $\p$ if and only if $X$ is analytic. Whereas Theorem~\ref{thm:christensen} says
that $X$ is strongly dominated by $\p$ if and only if $X$ is an absolute $G_{\delta}$ set. There is a big gap between
absolute $G_{\delta}$ sets and analytic sets. In this note we introduce a very small modification of the definition of $\p$
domination which characterizes Borel sets in the setting of Polish spaces. Let us first observe the following
simple fact.
\begin{factnonum}\label{fact:reform}
Space $X$ is $M$-dominated if and only if there exists a compact family $\{A_K : K \in \com(M) \}$
such that $\bigcup_{K \in \com(M)} A_K =X$ and $A_{K \cap L} \subseteq A_K \cap A_L$ for all $K, L \in \com (M)$.
\end{factnonum}
We show that making a simple change of replacing $\subseteq$ by $=$ leads to a characterization of Borel sets in the setting of Polish spaces. More precisely, we say that space $X$ is {\em strictly $M$-dominated} if there exists a compact family 
$\{A_K:K \in \com(M)\}$ such that $\bigcup_{K \in \com(M)} A_K =X$ and 
\[   K, L \in \com(M)  \implies  \left [ A_{K \cap L}  = A_K \cap A_L \right ].
\]

The main result of this note is to show the following theorem.
\begin{repeatthm} Let $X$ be a separable metric space and $X^*$ be its completion. $X$ is Borel in $X^*$ if and only if $X$ is strictly $\p$-dominated.
\end{repeatthm}
This new definition may characterize $K$-Lusin sets in Dieudonn\'e complete spaces. We have been unable to determine  this. However, when possible we prove auxiliary lemmas in as general settings as we know
how. 
\section{terminology and notation}
We use the standard terminology of general topology and descriptive set theory.
Our terminology concerning general topology may be found in \cite{engleking}, \cite{willard}. Our terminology
concerning descriptive set theory in abstract setting can be found in \cite{rogersjayne}.
For the convenience of the reader, we state below the specific definitions which we use in our proofs.

Let $X, Y$ be topological space and $\com (Y)$ the set of all compact subsets of  $Y$ endowed with the Hausdorff topology. 
Let $f: X \rightarrow \com (Y)$. Then, {\em $f$ is upper semicontinuous } if for every $x \in X$ 
and $V$ open in $Y$ with $f(x) \subseteq V$, there exists $U$ open in $X$ containing $x$ such that
$f(y) \subseteq V$ for all $y \in U$. 

A space  {\em  $X$ is $K$-analytic} if there exists an upper semicontinuous $f:\p \rightarrow \com (X)$
such that $X = \bigcup_{x \in X} f(x)$. If, moreover, $f$ has the property that $f(x) \cap f(y) = \emptyset$
whenever $x \neq y$, then {\em $X$ is said to be $K$-Lusin.} We recall the classical result that if $X$ is 
a separable metric space, then $X$ is $K$-Lusin if and only if $X$ is a Borel subset of $X^*$, the completion
of $X$. 
\section{main result}
In this section we give the proof of the main result.
\begin{lem}\label{lem:easydirection}
Let $X$ be a Tychnoff space which is  $K$-Lusin. Then, $X$ is strictly $\p$-dominated.
\end{lem}
\begin{proof}
Let $f: \p \rightarrow \com (X)$ witness the fact that $X$ is $K$-Lusin. For each
$K \subseteq \com (\p)$, let $A_K = \bigcup f(K)$. Clearly, $A_K \cap A_L = A_{K \cap L}$
and $X =\bigcup \{ K \in \com (\p): A_K\}$. To complete the proof, we need to verify that $A_K$
is compact.  Let $\g$ be  an open cover of $A_K$.  Hence $\g$ is also an open cover of $f(x)$
for each $x \in K$. As $f(x)$ is compact, we may choose a finite subcover $\V_x$ of $\g$ which covers $f(x)$.
Now, using the upper semicontinuity of $f$, we obtain an open set $U_x$ in $\p$, containing $x$, such that whenever $y \in U_x$,
we have that $f(y) \subseteq \bigcup \V_x$. As $K$ is compact, there exists $x_1, \ldots, x_n$ 
such that $\bigcup_{i=1}^n U_{x_i} \supseteq K$. Then, it follows that $\bigcup_{i=1}^n \V_{x_n}$ is
a finite subset of $\g$ which covers $A_K$. 
\end{proof}
\begin{lem}\label{lem:borel1}
Suppose that $X$ is a metric space which is  strictly dominated by $\p$. Then,  there exists an upper semicontinuous assignment $A:\com (\p) \rightarrow \com (X)$, given by $K \rightarrow A_K$, such that $X = \bigcup \{A_K: K \in \com (\p) \}$ and satisfies
the following condition:

\[ \tag{$\dagger$} \left [ K_1, K_2, \dots \in \p(\com) \right ] \implies  \left [\bigcap_{n=1}^{\infty} A_{K_n} = A _{\bigcap _{n=1}^{\infty} K_n} \right ] .\] 

\end{lem}
\begin{proof}Suppose that $X$ is strictly dominated by the family $\{F_K\}_{K \in \com (\p)}$. 
We define $A_{\emptyset} = F_{\emptyset}$. For each nonempty $K \in \com (\p)$ and $n \in \nat$, let $U_n (K)$ be the open set $\{t \in \p : d(t, K) < \frac{1}{n}\}$. We define $A_K$ in the following fashion:
\[ A_K = \bigcap_{n=1}^{\infty} \bigcup \left \{F_L: L \in \com(\p), K \subseteq L \subseteq U_n(K) \right \}.\] 
We now show that $\{A_K\}$ has the desired properties. \\

Let us first show that each $A_K$ is compact. As $A_{\emptyset}=F_{\emptyset}$, it is compact by hypothesis. Hence, let us assume that $K \neq \emptyset$.  As $X$ is metric, it suffices to show that every infinite sequence $\{x_n\}$ in $A_K$ has a subsequence whose limit is in $A_K$. Fix such a sequence. For each $n \in \nat$,
let $L_n$ be a compact set such that $ K \subseteq L_n \subset U_n (K) $ and $x_n \in F_{L_n}$.  Note that $\{ L_n \}$ converges to $K$ in the Hausdorff metric. Hence, for each $n$, $M_n = \cup _{i=n}^{\infty} L_i$
is compact subset of $U_n(K)$. Moreover,  $\{x_n, x_{n+1}, \ldots \} \subseteq F_{M_n}$ for all $n$. As $F_{M_1}$ is compact, there is $p$ which is the limit of some subsequence of $\{x_n\}$. As $\{M_n\}$ are monotonic and 
$\{F_{M_n}\}$ compact,  we have that $p \in F_{M_n}$ for all $n$. Therefore, $p \in A_K$, completing the proof of the compactness of $A_K$.\\

We next observe that $F_K \subseteq A_K \subseteq X$. Therefore, $\bigcup \{A_K: K  \in \com(\p)\} = X$. \\

We next show that \[ \tag{*} \left [ K, L  \in \p(\com) \right ] \implies  \left [ A_{K \cap L} = A_K \cap A_L \right ].\]
First, observe that 
\[  \left [ K \subseteq L  \right ] \implies  \left [ F_K \subseteq F_L\right ],  \mbox{ and }\]
\[
\left [K \subseteq K_1 \subseteq U_n(K) \right] \implies \left [  L \subseteq K_1 \cup L \subseteq U_n(L) \right].
\]
These two facts together imply that
\[ \tag{**} \left [ K \subseteq L  \right ] \implies  \left [ A_K \subseteq A_L\right ] .\]

Therefore, 
 \[ \left [ K, L  \in \p(\com) \right ] \implies  \left [ A_{K \cap L} \subseteq A_K \cap A_L \right ].\]
 
 Let us prove the reverse containment now.  Let $p \in A_K \cap A_L$. We wish to show that $p \in A_{K \cap L}$.
 Let us first consider the case $K \cap L = \emptyset$.
 In this case, there exists $n$ such that $U_n(K) \cap U_n(L) = \emptyset$.  As $p \in A_{K} \cap A_{L}$,
there exists $K_1 \in U_n(K), L_1 \in U_n(L)$ such that $p \in F_{K_1} \cap F_{L_1}$.  As $F_{K_1} \cap F_{L_1} = F_{K_1 \cap L_1} = F_{\emptyset} =A _{\emptyset}= A_{K \cap L}$, we have that $p \in A_{K \cap L}$, proving the containment in this case.

 Let us next consider the case that $K \cap L \neq \emptyset$. For each positive integer $n$, let $K_n , L_n \in \p (\com)$ be such that
\begin{itemize}
\item $K \subseteq K_n \subseteq U_n(K)$, $L \subseteq L_n \subseteq U_n(L)$, and 
\item $p \in F_{K_n}$, $p \in F_{L_n}$. 
\end{itemize}
As $\{K_n\}$, $\{L_n\}$ converge, respectively, to $K, L$ in the Hausdorff metric and $K \subseteq K_n$ and $L \subseteq L_n$, we have that $\{K_n \cap L_n\}$ converge to $K \cap L$ in
the Hausdorff metric. Hence every  $U_n(K \cap L)$ contains some $K_m \cap L_m$. Since $F_{K_m} \cap F_{L_m} = F_{K_m \cap L_m}$,
we have that $p \in F_{K_m \cap L_m}$. Therefore, by the definition of $A_{K\cap L}$ we have that $p \in A_{K \cap L}$, and completing
the proof of the containment \[ \left [ K, L  \in \p(\com) \right ] \implies  \left [  A_K \cap A_L   \subseteq  A_{K \cap L} \right ]. \]\

We next show that $(\dagger)$ holds. Let  $K_1, K_2, \dots$  be elements in $ \p(\com) $. By $(**)$ we have that 
\[  \left [ A _{\bigcap _{n=1}^{\infty} K_n}  \subseteq  \bigcap_{n=1}^{\infty} A_{K_n}  \right  ] . \] 
To prove the reverse containment, let $p \in \bigcap_{n=1}^{\infty} A_{K_n}$. We first observe that by $(*)$ and induction we have that 
\[  \left [ A _{\bigcap _{n=1}^{m} K_n} = \bigcap_{n=1}^{m} A_{K_n}  \right  ] \] 
for all $m \in \nat$.   We again consider two cases. The first case is that $\cap_{n=1}^{\infty} K_n = \emptyset$. 
As $K_n$'s are compact, we have that there exists $m$ such that $\cap_{n=1}^m K_n = \emptyset$.
Then,
\[ p \in \bigcap_{n=1}^{\infty} A_{K_n}\subseteq \bigcap_{n=1}^{m} A_{K_n}=A _{\bigcap _{n=1}^{m} K_n}=A_{\emptyset}
= A _{\bigcap _{n=1}^{\infty} K_n},
\] 
completing the proof in this case.  Now let us consider the case $\bigcap _{n=1}^{\infty} K_n \neq \emptyset$.
By compactness of $K_n$'s we have that $\bigcap _{n=1}^{m} K_n \neq \emptyset$ for all $m$.
As $p \in \bigcap_{n=1}^{m} A_{K_n}=A _{\bigcap _{n=1}^{m} K_n}$ and $\bigcap _{n=1}^{m} K_n \neq \emptyset$ we may choose $L_m  \in \com(\p)$ such that $\bigcap _{n=1}^{m} K_n \subseteq L_m \subseteq U_m( \bigcap _{n=1}^{m} K_n)$
and $p \in F_{L_m}$.  We note that $\{\bigcap_{n=1}^mK_n\}$ converges in Hausdorff metric to $\bigcap_{n=1}^{\infty}K_n$
as $m \rightarrow \infty$.
As $\bigcap _{n=1}^{m} K_n \subseteq L_m \subseteq U_m( \bigcap _{n=1}^{m} K_n)$, we have that $\{L_m\}$
converges to $\bigcap_{n=1}^{\infty}K_n$ in the Hausdorff metric as well.  Hence, we have that
for every $m \in \nat$ there is $m' \in \nat$ such that ${L_{m'} }\subseteq U_m(  \bigcap_{n=1}^{\infty} {K_n})$. 
As $p \in F_{L_m'}$, we have that $p \in A _{\bigcap _{n=1}^{\infty} K_n}$, completing the proof.

Finally, let us show that the assignment $K \rightarrow A_K$ is an upper semicontinuous map. To this end, let $K \in \com (\p)$ 
and $V$ open in $X$ such that be such that $A_K \subseteq V$.  We need to show that there exists some $\delta  >0$
such that if the Hausdorff distance between $K$ and $L$ is less than $\delta$, then $A_L \subseteq V$. We will show
something stronger. Namely, there is an open set $U$ in $\p$ with $K \subseteq U$ such that if $L \in \com(\p)$ and $L \subseteq U$,
then $F_L \subseteq V$.  This suffices as for sufficiently large $n$, we have that $U_n(F_L) \subseteq U$, implying that
$A_{L} \subseteq V$. To obtain a contradiction, assume that there is no such $U$ and for each $n \in N$, let $K_n \in \com(\p)$ be such that $K_n \subseteq U_n (K)$ and $F_{K_n} \nsubseteq V$. We note that $\{K_n\}$ converges to $K$ in the Hausdorff metric and hence $L_n = \cup_{i=n}^{\infty} K _i \cup K$ is a compact subset of $U_n(K)$ containing $K$. As $F_{L_n}  \setminus V \neq \emptyset$
we have that $A_{K} \setminus V \neq \emptyset$, yielding a contradiction. 
\end{proof}
\begin{lem}\label{lem:borel2}
Suppose $X$ is a separable metric space and $X^*$ is the completion of $X$. If $X$ is strictly dominated by $\p$, then $X$ is a Borel subset of  $X^*$. 
\end{lem}
\begin{proof} Let $K \rightarrow A_K$ be the assignment from Lemma~\ref{lem:borel1}. As the assignment $K \rightarrow A_K$ is upper semicontinuous, it is Borel and hence its graph 
$ \{(K, A_K) : K \in  \com(\p) \} $ is a Borel subset of  $\com(\p) \times \com (X^*)$. Furthermore, $ \{(x, L) : x \in L, L \in \com(X^*)\}$ 
is a closed subset of $X^* \times \com (X^*)$. Therefore, we have that
\[ \{ (K, x, A_K): K \in \com(\p), x \in A_K\}
\]
is a Borel subset of $\com(\p) \times X^* \times \com (X^*)$. Moreover, 
\[\T = \{ (K, x): K \in \com(\p), x \in A_K\}\] is the 1-1 projection of this set and hence is Borel.

  Let $\{ B_1, B_2, \ldots\} $ be a basis of $\p$ consisting of nonempty clopen sets. For each $n$,
let \[ \T_n = \{ (K, x): K \in \com(\p), K \cap  B_n \neq \emptyset,  x \in A_{K \setminus B_n} \}.
\]
Let us fix $n$ for the moment. The assignment $K \rightarrow K\setminus B_n$ is continuous
and the assignment $K \rightarrow A_K$ is upper semicontinuous. Hence, the assignment
$K \rightarrow A_{K \setminus B_n}$ is Borel.  Moreover, $\{K \in \com (\p): K \cap B_n \neq \emptyset \}$
is closed. Hence, the set
\[ \{ (K, x, A_{K \setminus B_n}): K \in \com(\p), K \cap  B_n \neq \emptyset,  x \in A_{K \setminus B_n} \}
\]
is Borel. Now $\T _n$ is simply 1-1 project of this set and hence itself is Borel.
Now let us consider the Borel set
\[ \s= \T \setminus \left (  \cup_{n=1}^{\infty} \T _n \right ).
\] 
We first observe that if $(K, x) \in \s$, then $x \in A_K$ but $x \notin A_L$ for any proper subset $L$ of $K$. Indeed, this is true for otherwise we can find $n$ such that $K \cap B_n \neq \emptyset$ and $L \subset K \setminus B_n$. For this $n$, we would have that $(K,x) \in \T_n$, contradicting that $(K,x) \in \s$.

 To conclude that $X$ is Borel, it will suffice to show that $\pi_2$, the projection onto the second coordinate, is 1-1 on $\s$ and $\pi_2(S) =X$. We first show that $\pi_2$ is 1-1.
To obtain a contradiction, assume that $(K,x), (L,x) \in \s$ with $K \neq L$. Then,
$x \in A_K \cap A_L = A_{K \cap L}$. As $K \neq L$, $K \cap L$ is a proper subset of either $K$ or $L$,
yielding a contradiction. Hence, $\pi_2$ is 1-1 on $\s$. 

To observe that $\pi_2(S) =X$, let $p \in X$. Let $\g _p = \{ L \in \com (\p): p \in A_L\}$ and let $K =\bigcap _{L \in \g _p} A_L$.
We claim that $p \in A_K$. We consider two cases, $K = \emptyset$ and $ K \neq \emptyset$. If $K = \emptyset$, then
there exists $K_1, \ldots, K_n \in \g _p$ such that $\bigcap_{i=1}^n K_i = \emptyset$. Then,
\[ p \in \bigcap_{i=1}^n A_{K_i} = A_{\bigcap_{i=1}^n K_i} = A_{\emptyset} =A_K.
\]
Let us now consider the case $K \neq \emptyset$. Then there exists a sequence of compact set $\{K_n\}_{n=1}^{\infty}$
in $\g_p$ such that $K = \cap_{n=1}^{\infty}K_n$. (This is so because $\p$ is hereditarily Lindel\"of as it is a separable metric 
space. Hence some countable subcollection
of $\{L^c: L \in \g_p\}$ covers $K^c$.)
By Condition $(\dagger)$, we have that 
\[ p \in \bigcap_{n=1}^{\infty}A_{K_n} = A_{\bigcap_{n=1}^{\infty}{K_n}} = A_K.
\]
Therefore, $(K, p) \in \T$. By the fashion in which $K$ is defined, we have that $(K, p)$
is not in $\T_n$ for any $n$. Hence $(K, p) \in \s$, completing the proof.
\end{proof}
\begin{thm}\label{thm:repeat} Let $X$ be a separable metric space. Then, $X$ is strictly $\p$-dominated if and only if $X$ is a Borel set of of $X^*$, the completion of $X$. 
\end{thm}
\begin{proof} If $X$ is a Borel subset of $X^*$, then $X$ is a $K$-Lusin set and we obtain that $X$ is strictly $\p$-dominated  by Lemma~\ref{lem:easydirection}.  The other direction follows from Lemma~\ref{lem:borel2}.
\end{proof}

\end{document}